\documentclass[a4paper,12pt,twoside]{article}
\usepackage[ansinew]{inputenc}
\usepackage{amsfonts}
\usepackage{latexsym}
\usepackage{amsmath}
\usepackage{amssymb}
\usepackage{eepic}

\newcommand{\veps}{\varepsilon}
\newcommand{\R}{\mathbb{R}}

\newcommand{\C}{\mathbb{C}}
\newcommand{\N}{\mathbb{N}}

\newcommand{\D}{\mathbb{D}}

\newcommand{\co}{\operatorname{co}}

\newtheorem{lettertheorem}{Theorem}

\newtheorem{defin}{Definition}

\newtheorem{theorem}[defin]{Theorem}
\newtheorem{exa}{Example}
\newenvironment{example}{\begin{exa}\rm}{\end{exa}}
\newtheorem{lemma}[defin]{Lemma}
\newtheorem{corollary}[defin]{Corollary}

\newtheorem{rem}{Remark}
\newenvironment{remark}{\begin{rem}\rm}{\end{rem}}
\newenvironment{proof}
{\noindent{\it Proof.}}{\hfill $\Box$\par\vspace{2.5mm}}

\numberwithin{equation}{section}

\makeatletter
\renewcommand{\ps@myheadings}{%
\renewcommand{\@evenhead}%
{{\rm\thepage}\hfil{\sc Heittokangas, Ishizaki, Laine, Tohge}\hfil}%
\renewcommand{\@oddhead}%
{\hfil{{\sc Exponential polynomials in the oscillation theory}\hfil{\rm\thepage}}}%
\renewcommand{\@evenfoot}{}%
\renewcommand{\@oddfoot}{}%
}\makeatother 

\title{\bf\Large Exponential polynomials in the\\ oscillation theory}
\author{J.~Heittokangas, K.~Ishizaki, I.~Laine, K.~Tohge}
\date{}
\begin{document}
\maketitle

\begin{abstract}
Supposing that $A(z)$ is an exponential polynomial of the form
    $$
    A(z)=H_0(z)+H_1(z)e^{\zeta_1z^n}+\cdots +H_m(z)e^{\zeta_mz^n},
    $$
where $H_j$'s are entire and of order $<n$, it is demonstrated that the function
$H_0(z)$ and the geometric location of the leading coefficients $\zeta_1,\ldots,\zeta_m$ play a key role in the oscillation of solutions
of the differential equation $f''+A(z)f=0$. The key tools consist of value distribution
properties of exponential polynomials, and elementary properties of the
Phragm\'en-Lindel\"of indicator function. In addition to results in the whole complex
plane, results on sectorial oscillation are proved.

\medskip
\noindent
\textbf{Key Words:} Exponent of convergence, exponential polynomial, differential equation, oscillation theory, Phragm\'en-Lindel\"of indicator.

\medskip
\noindent
\textbf{2010 MSC:} Primary 34M10; Secondary 34M03.
\end{abstract}

\renewcommand{\thefootnote}{}
\footnotetext[1]{Financially supported by the Academy of Finland \#268009, the JSPS KAKENHI Grant
Numbers JP25400131 and JP16K05194, and the Discretionary Budget (2017) of the President of the 
Open University of Japan.}

\section{Introduction}\label{intro-section}

In the 1982 paper \cite{Bank-Laine} due to Bank and Laine a sequence of results is obtained on the oscillation of
solutions of the differential equation
    \begin{equation}\label{lde2}
    f''+A(z)f=0,
    \end{equation}
where $A(z)$ is an entire function. This paper prompted an extensive amount of investigations over the next three decades,
not the least due to what since then has been called the Bank--Laine conjecture: \emph{If $A(z)$ is transcendental entire with order
$\rho (A) \not\in \N$, then $\max\{\lambda (f_1),\lambda (f_2)\} = \infty$ for linearly independent solutions $f_1,f_2$ of \eqref{lde2}.}
Here $\lambda(g)$ stands for the exponent of convergence of zeros of $g$.

Recently, the Bank-Laine conjecture has been disproved by Bergweiler and Eremenko \cite{BE1,BE2}.
The construction of the coefficient $A(z)$ in \cite{BE1} is really sophisticated, but relies on exponential
polynomials of the form $P_m(e^z)$, where $P_m(w)$ is the $2m$th degree Maclaurin polynomial of the function $e^{-w}$.

\begin{lettertheorem}[\cite{BLL0}]\label{BLL0-thm}
Suppose that $A(z)=e^z-K$, where $K\in\C$. If \eqref{lde2} possesses a nontrivial solution $f$
with exponent of convergence $\lambda(f)<\infty$, then $K=q^2/16$, where $q$ is an odd positive integer.
Conversely, if $K$ is as above, then \eqref{lde2} with $A(z)=e^z-K$ has two linearly independent
solutions $f_1,f_2$ for which $\lambda(f_1)=0=\lambda(f_2)$ if $q=1$ and $\lambda(f_1)=1=\lambda(f_2)$ otherwise.
\end{lettertheorem}

\begin{example}
Regarding Theorem~\ref{BLL0-thm}, for $j=1,2$, the functions
    \begin{eqnarray*}
    f_j(z) &=& \exp\left((-1)^j2i\exp\left(\frac{z}{2}\right)-\frac{z}{4}\right)\\
    f_j(z) &=& \left((-1)^j\frac{i}{2}\exp\left(-\frac{z}{2}\right)+1\right)\exp\left((-1)^j2i\exp\left(\frac{z}{2}\right)-\frac{z}{4}\right)
    \end{eqnarray*}
are linearly independent solutions of \eqref{lde2} in the cases $q=1$ and $q=3$, respectively. The coefficient $A(z)=e^z-q^2/16$ is clearly
$2\pi i$-periodic, which gives raise to $f_1(z)=if_2(z+2\pi i)$ in both cases.
\end{example}

The following generalization of Theorem~\ref{BLL0-thm} is commonly known as the ''$\frac{1}{16}$-theorem''.
An improvement of this classical result will be given in Theorem~\ref{referee-thm} below.

\begin{lettertheorem}[\cite{BLL}]\label{BLL-thm}
Suppose that $A(z)=e^{P(z)}+Q(z)$, where $P(z)$ is a polynomial of degree $n\geq 1$ and $Q(z)$ is an
entire function of order $\rho(Q)<n$. If \eqref{lde2} possesses a nontrivial solution $f$
with exponent of convergence $\lambda(f)<n$, then $f$ has no zeros and $Q(z)$ reduces to a
polynomial of the form
    $$
    Q(z)=-\frac{1}{16}P'(z)^2+\frac{1}{4}P''(z).
    $$
Moreover, \eqref{lde2} admits in this case
a zero-free solution base.
\end{lettertheorem}

This development raised further interest in the case when the coefficient $A(z)$ contains more than one
exponential term.

\begin{lettertheorem}[\cite{Ishizaki, Ishizaki-Tohge}]\label{Ishizaki-Tohge}
Suppose that $A(z)=e^{P_1(z)}+e^{P_2(z)}+Q(z)$, where $P_1(z)=\zeta_1z^n+\cdots$ and
$P_2(z)=\zeta_2z^m+\cdots $ are nonconstant polynomials such that $e^{P_1(z)}$ and
$e^{P_2(z)}$ are linearly independent, and $Q(z)$ is an entire function of order
$<\max\{n,m\}$. Let $f$ be a nontrivial solution of \eqref{lde2}.
If $n\neq m$, then $\lambda(f)=\infty$, while if $n=m$, we have the following cases:\\[-20pt]
\begin{itemize}
\item[{\rm (a)}] If $\zeta_1=\zeta_2$, then $\lambda(f)\geq n$.\\[-22pt]
\item[{\rm (b)}] If $\zeta_1\neq\zeta_2$ and $\zeta_1/\zeta_2$ is non-real, then $\lambda(f)=\infty$.\\[-22pt]
\item[{\rm (c)}] If $0<\zeta_1/\zeta_2<1/2$, then $\lambda(f)\geq n$.\\[-22pt]
\item[{\rm (d)}] If $3/4<\zeta_1/\zeta_2<1$ and $Q(z)\equiv 0$, then $\lambda(f)\geq n$.
\end{itemize}
\end{lettertheorem}

The assumption $0<\zeta_1/\zeta_2<1$ means geometrically that the points $\zeta_1,\zeta_2$
are on the same half-line emanating from the origin. In the case that $0,\zeta_1,\zeta_2$ are collinear but
$\zeta_1,\zeta_2$ are on the opposite sides of the origin, we will show in Corollary~\ref{opposite-sides-thm},
Section~\ref{elementary-sec}, that $\lambda(f)=\infty$.

As for the constants 1/2 and 3/4 in Theorem~\ref{Ishizaki-Tohge}, there are examples of zero-free solutions in the
following two cases:\\[-20pt]
\begin{itemize}
\item[(1)] $\zeta_1/\zeta_2=1/2$;\\[-22pt]
\item[(2)] $\zeta_1/\zeta_2=3/4$ and $Q(z)\not\equiv 0$.\\[-20pt]
\end{itemize}
See \cite{Bank-Laine} for (1) with $Q(z)\equiv 0$ and \cite{Ishizaki} for (2). Examples~\ref{half-ex}, \ref{power2-ex}
and \ref{rationalnumbers-ex} below deal with (1) in the case when $Q(z)\not\equiv 0$.

In this research we assume that $A(z)$ is a general exponential polynomial of the form
    \begin{equation}\label{exp-poly}
    A(z)=P_1(z)e^{Q_1(z)}+\cdots +P_k(z)e^{Q_k(z)},
    \end{equation}
where $P_j$'s and $Q_j$'s are polynomials in $z$, and obtain several results on the oscillation
of solutions of \eqref{lde2}. Following Steinmetz \cite{Stein}, the exponential polynomial $A(z)$ in \eqref{exp-poly}
can be written in the normalized form
    \begin{equation}\label{exp-poly2}
    A(z)=H_0(z)+H_1(z)e^{\zeta_1z^n}+\cdots +H_m(z)e^{\zeta_mz^n},
    \end{equation}
where $H_j(z)$'s are either exponential polynomials of order $<n$ or ordinary polynomials in $z$,
the leading coefficients $\zeta_j$ are pairwise distinct, and $m\leq k$.

In the further sections we pursue in the spirit of Theorem~\ref{Ishizaki-Tohge}. For example, we show that arbitrarily many new
exponential terms may be added to the exponential polynomial coefficient $A(z)$ without violating the conclusion
in Theorem~\ref{Ishizaki-Tohge}(d),
for as long as certain geometry of the leading coefficients does not change. Parallel
discussions are already available in \cite{Tu}, where the approach is different from ours.

The proofs in this paper rely on properties of exponential polynomials
\cite{Stein} and of the Phragmen-Lindel\"of indicator function \cite{Levin1}. These properties will be reviewed in Section~\ref{elementary-sec}.

\section{Tools and elementary observations}\label{elementary-sec}

We assume the fundamental results and standard notation in Nevanlinna theory \cite{Hayman, Laine, YY}.
In addition, the convex hull of a finite set $W\subset\C$,
denoted by $\co (W)$, is the intersection of finitely many closed half-planes each including $W$, and hence
$\co (W)$ is either a compact polygon or a line segment. We denote the circumference of
$\co (W)$ by $C(\co (W))$. Concerning the exponential polynomial $A(z)$ in \eqref{exp-poly2}, we denote
$W=\left\{\bar{\zeta}_1,\ldots ,\bar{\zeta}_m\right\}$ and $W_0=W\cup\{0\}$.

\begin{lettertheorem}[\cite{Stein}]\label{Stein-thm}
Let $A(z)$ be given by \eqref{exp-poly2}. Then
    \begin{equation}\label{characteristic}
    T(r,A)=C(\co (W_0))\frac{r^n}{2\pi}+o(r^n).
    \end{equation}
If $H_0(z)\not\equiv 0$, then
    \begin{equation}\label{proximity}
    m\left(r,\frac1A\right)=o(r^n),
    \end{equation}
while if $H_0(z)\equiv 0$, then
    \begin{equation}\label{counting}
    N\left(r,\frac1A\right)=C(\co(W))\frac{r^n}{2\pi}+o(r^n).
    \end{equation}
If $H_0(z)\not\equiv 0$ is a polynomial, then \eqref{proximity} can be replaced with
   \begin{equation}\label{proximity2}
    m\left(r,\frac1A\right)=O(\log r).
    \end{equation}
\end{lettertheorem}

The Phragm\'en-Lindel\"of indicator \cite{Levin1} of an entire function $f$ of order $\rho>0$ is given by
    \begin{equation}\label{indicator-def}
    h_f(\theta)=\limsup_{r\to\infty}r^{-\rho}\log |f(re^{i\theta})|.
    \end{equation}
For example, if $f(z)=\exp\left(wz^\rho\right)$, where $w\in\C\setminus\{0\}$ and
$\rho$ is a positive integer, then $h_f(\theta)=\Re(we^{i\rho\theta})$.
It is clear that $h_f$ is $2\pi$-periodic in general. If $h_f$ is
bounded, then it is continuous. For a product and a sum of entire functions, we have
$h_{f_1f_2}(\theta)\leq h_{f_1}(\theta)+h_{f_2}(\theta)$ and $h_{f_1+f_2}(\theta)\leq \max\{h_{f_1}(\theta),h_{f_2}(\theta)\}$.
Moreover, if $h_{f_1}(\theta)\neq h_{f_2}(\theta)$ for some $\theta$, then
    \begin{equation}\label{sum}
    h_{f_1+f_2}(\theta)= \max\{h_{f_1}(\theta),h_{f_2}(\theta)\}.
    \end{equation}

Suppose that $A(z)$ is given by \eqref{exp-poly2}. Since $A(z)$ is of bounded type \cite[Satz~1]{Stein}, its
indicator $h_A(\theta)$ is continuous. By the proof of \cite[Lemma~3]{Stein}, the indicator of
precisely one exponential term $H_j(z)e^{\zeta_jz^n}$ at the time dominates the indicators of the other
exponential terms in $A(z)$, except possibly for finitely many angles $\theta\in[-\pi,\pi)$. Thus,
by \eqref{sum} and continuity, we deduce that
    $$
    h_A(\theta)=\max_{j\leq m}\{\Re\left(\zeta_je^{in\theta}\right)\},
    $$
where $\zeta_0=0$ if $H_0(z)\not\equiv 0$ and $j\geq 1$ otherwise.

Recall from \cite{Stein} that if $f$ is an exponential polynomial of order $n$, then
''limsup'' in \eqref{indicator-def} can be replaced with ''lim'', that is,
    $$
    h_{f}(\theta)=\lim_{r\to\infty}r^{-n}\log |f(re^{i\theta})|
    $$
holds for all except possibly for finitely many values of $\theta$ on $[-\pi,\pi)$. Hence for every
$\theta\in [-\pi,\pi)$ with at most finitely many exceptions, we have
    \begin{equation}\label{f-behavior}
    \log |f(re^{i\theta})|=(h_{f}(\theta)+o(1))r^{n},\quad r\to\infty.
    \end{equation}
This is a key property for showing that exponential polynomials are of completely regular
growth, see \cite[Lemma~1.3]{GOP}.

For an entire function $f$ of finite order and of bounded type, we may write $f'=(f'/f)\cdot f$ and use \cite[Corollary~1]{Gundersen} to obtain
$h_{f'}(\theta)\leq h_f(\theta)$ for almost all $\theta$. By continuity, it follows that
    \begin{equation}\label{derivative-indicator}
    h_{f'}(\theta)\leq h_f(\theta)
    \end{equation}
for every $\theta$. If $f$ is meromorphic in $\C$ of order $\rho$  with
finitely many poles, then the expression $h_f(\theta)$ is still well-defined for every $\theta$.
Recall that the poles of a meromorphic function of finite order can be enclosed in a
collection of discs commonly known as an $R$-set, see \cite[Chapter~5]{Laine}. Therefore, even if $f$ has
infinitely many poles, the indicator $h_f(\theta)$ is well-defined
for almost all $\theta$.

The proof of Theorem~\ref{Ishizaki-Tohge}(d) in \cite{Ishizaki} is based on the following result.

\begin{lettertheorem}[\cite{Ishizaki-Tohge}]\label{K-thm-IT}
Let $K>4$, and let $A(z)$ be a transcendental entire function of order $\rho(A)$ such that
    \begin{equation}\label{K}
    K\overline{N}\left(r,\frac{1}{A}\right)\leq T(r,A)+S(r,A).
    \end{equation}
Then every nontrivial solution $f$ of \eqref{lde2} satisfies $\lambda(f)\geq\rho(A)$.
\end{lettertheorem}

\begin{remark}
(1) It is apparent that the assumption \eqref{K} can be replaced with
    \begin{equation}\label{theta}
    \Theta(0,A):=1-\limsup_{r\to\infty}\frac{\overline{N}\left(r,\frac{1}{A}\right)}{T\left(r,A\right)}>1-\frac{1}{K}.
    \end{equation}
Since the sum of ramifications $\Theta(c,g)$ over $c\in\C$ for any entire function $g$ is at most one by \cite[Corollary 2.5.5]{Laine},
the function $A(z)$ cannot have completely ramified values other than possibly zero by \eqref{theta}.

(2) As a special case of the original result due to Bank and Langley, \cite[Theorem~5.7]{Laine} asserts the following:
If $A(z)=e^{P(z)}$, where $P(z)$ is a nonconstant polynomial, then every nontrivial solution $f$ of \eqref{lde2}
satisfies $\lambda(f)=\infty$. The proof is rather involved even for this special case. A slightly weaker result
$\lambda(f)\geq\rho(A)$ follows trivially from Theorem~\ref{K-thm-IT}.
\end{remark}

We state the following obvious consequence of Theorems~\ref{Stein-thm} and \ref{K-thm-IT}.

\begin{corollary}\label{n-thm}
Let $A(z)$ be an exponential polynomial of the form
    \begin{equation}\label{A}
    A(z)=H_1(z)e^{\zeta_1z^n}+\cdots +H_m(z)e^{\zeta_mz^n},
    \end{equation}
where the functions $H_j(z)$ are either
exponential polynomials of order $<n$ or ordinary polynomials in $z$. Suppose that
    \begin{equation}\label{perimeters}
    C(\co(W_0))>4C(\co (W)).
    \end{equation}
Then every nontrivial solution $f$ of \eqref{lde2} satisfies $\lambda(f)\geq n$.
\end{corollary}

\begin{remark}
(1) Suppose that $W=\left\{\bar{\zeta}_1,\bar{\zeta}_2\right\}$ and $\zeta_1=\rho\zeta_2$ for $0<\rho<1$
as in Theorem~\ref{Ishizaki-Tohge}. Then the assumption \eqref{perimeters} reads as
    $$
    2|\zeta_2|>2\cdot 2(|\zeta_2|-|\zeta_1|)=4(1-\rho)|\zeta_2|,
    $$
from which we get $\rho>1/2$. Thus Corollary~\ref{n-thm}
generalizes Theorem~\ref{Ishizaki-Tohge}(d), and the result is sharp in the sense that equality in
\eqref{perimeters} cannot hold \cite{Ishizaki}.

(2) We may add more exponential terms of order $n$ to the coefficient $A(z)$ without affecting the assertion,
provided that the conjugates of the leading coefficients of the added exponential terms belong to $\co (W)$. Indeed,
such an ''addition'' has no affect on \eqref{perimeters}. In particular, we could replace
$A(z)=e^{P_1(z)}+e^{P_2(z)}$ in Theorem~\ref{Ishizaki-Tohge}(d) with $A(z)=e^{P_1(z)}+\cdots+e^{P_k(z)}$,
where $P_j(z)$'s are all polynomials of the same degree $n$, and the leading coefficients of
$P_3(z),\ldots,P_k(z)$ are on the interval $[\zeta_1,\zeta_2]$.

(3) If the convex hull $\co (W)$ has a large circumference, then it needs to be located sufficiently far away from
the origin in order for \eqref{perimeters} to hold.
\end{remark}

As a further motivation to this study we re-formulate the perturbation
result \cite[Theorem~3.1]{BLL}. This result shows that the $H_0(z)$-term of the exponential
polynomial $A(z)$ in \eqref{exp-poly2} plays a role in the oscillation theory.

\begin{lettertheorem}[\cite{BLL}]\label{perttu-thm}
Let $A(z)$ be an exponential polynomial of order $n$, and let $f_1,f_2$ be two linearly independent
solutions of \eqref{lde2} with $\max\{\lambda(f_1),\lambda(f_2)\}<n$. Then, for any entire function
$B(z)\not\equiv 0$ of order $\rho(B)<n$, any two linearly independent solutions $g_1,g_2$ of the
differential equation
    $$
    g''+(A(z)+B(z))g=0
    $$
satisfy $\max\{\lambda(g_1),\lambda(g_2)\}\geq n$.
\end{lettertheorem}

\begin{remark}
Prior to Theorem~\ref{perttu-thm}, it was shown in \cite[Corollary~1]{BLL0} that if $A(z)$ is an
exponential polynomial of the form \eqref{exp-poly} with non-constant polynomials $Q_j(z)$, then
$\max\{\lambda(f_1),\lambda(f_2)\}=\infty$ for any linearly independent solutions $f_1,f_2$ of \eqref{lde2}.
In other words, for the assumption $\max\{\lambda(f_1),\lambda(f_2)\}<n$ in Theorem~\ref{perttu-thm}
to be possible, at least one of the polynomials $Q_j(z)$ must be constant.
\end{remark}

The situation in the previous remark may happen.
In fact, as a straight forward consequence of \cite[Theorem 5.6]{Laine}, we next state that
\eqref{lde2} may have a zero-free solution base, although this is not very typical:

\begin{corollary}[\cite{Bank-Laine,Laine}]
Let $A(z)$ be an exponential polynomial. Then \eqref{lde2} has a zero-free solution base if and only if
$A(z)$ can be represented in the form
    \begin{equation}\label{special-A}
    -4A(z)=e^{2P(z)}+P'(z)^2-2P''(z),
    \end{equation}
where $P(z)$ is some nonconstant polynomial. In particular, $A(z)$ in \eqref{special-A} has only one
exponential term.
\end{corollary}

Theorem~\ref{Ishizaki-Tohge} does not address the case where $n=m$ and the points
$0,\zeta_1,\zeta_2$ are collinear such that $\zeta_1,\zeta_2$ are on the opposite sides of the origin.
Without loss of generality, we may assume that $\zeta_1>0$
and $\zeta_2<0$. This divides the complex plane into $2n$ sectors in such a way
that in every sector precisely one of the indicators for $e^{\zeta_1z^n}$ and
$e^{\zeta_2z^n}$ is
positive and the other one is negative. On the boundaries of these sectors both
indicators vanish. Therefore the following result is an immediate consequence
of \cite[Theorem~4.3]{BLL}.

\begin{corollary}\label{opposite-sides-thm}
Suppose that $A(z)=P_1(z)e^{\zeta_1z^n}+P_2(z)e^{\zeta_2z^n}+Q(z)$, where $P_1(z),P_2(z)\not \equiv 0$ and $Q(z)$ are polynomials,
$\zeta_1>0$ and $\zeta_2<0$.
Then every nontrivial solution $f$ of \eqref{lde2} satisfies $\lambda(f)=\infty$.
\end{corollary}

If $A(z)$ has more than a pair of collinear leading coefficients on the opposite sides of the origin, then
\eqref{lde2} might have zero-free solutions.

\begin{example}
If $A(z)=-e^{-z}-e^{-2z}-e^z-e^{2z}+2$, then $f(z)=\exp\left(e^{-z}+e^z\right)$ solves \eqref{lde2}.
If $A(z)=-e^{-z}-e^{-2z}+4e^z-4e^{2z}-4e^{4z}$, then $f(z)=\exp\left(e^{-z}+e^{2z}\right)$ solves \eqref{lde2}.
Here $Q(z)\equiv 2$ and $Q(z)\equiv 0$, respectively. In both cases $h_A(\pm\pi/2)=0$ and $h_A(\theta)>0$ elsewhere
on $[-\pi,\pi)$.
\end{example}

\section{Estimates of Frank-Hennekemper type}\label{FH-sec}

This section contains estimates for the logarithmic derivative of meromorphic functions, which 
will be used in the oscillation theory later on.

The proof of Theorem~\ref{K-thm-IT} in \cite{Ishizaki-Tohge} is based on the following estimate originally due to Frank and Hennekemper.

\begin{lettertheorem}[\cite{FH}]\label{FH-original}
Let $f$ be a transcendental meromorphic function, and let $k\geq 2$ be an integer. Then
    \begin{equation}\label{FH-estim}
    m\left(r,\frac{f'}{f}\right)\leq 2\overline{N}\left(r,\frac{1}{f}\right)+2\overline{N}\left(r,\frac{1}{f^{(k)}}\right)+S\left(r,\frac{f'}{f}\right).
    \end{equation}
\end{lettertheorem}

The original proof of \eqref{FH-estim}  is somewhat involved. A slightly simplified proof is available in \cite{FHV}. In the particular case when
$f$ is a solution of \eqref{lde2} and $k=2$, the estimate \eqref{FH-estim} reads as
    \begin{equation}\label{reads-as}
     m\left(r,\frac{f'}{f}\right)\leq 2\overline{N}\left(r,\frac{1}{f}\right)
    +2\overline{N}   \left(r,\frac{1}{fA}\right)+S\left(r,\frac{f'}{f}\right).
    \end{equation}

We proceed to study estimates of the above type in the case when $f$ solves \eqref{lde2} by making a direct use of \eqref{lde2}.
This simplifies the proof and leads to slightly different conclusions than above.

\begin{lemma}\label{FH-lem}
Let $f$ be a nontrivial solution of \eqref{lde2}, where $A(z)$ is a transcendental entire function. Then
    \begin{equation}\label{FH-estim2}
    m\left(r,\frac{f'}{f}\right)\leq \overline{N}\left(r,\frac{1}{ff'A}\right)+S\left(r,\frac{f'}{f}\right).
    \end{equation}
\end{lemma}

\begin{proof}
Denote $g=f'/f$ for short. Then $-A(z)=g^2+g'$ by \eqref{lde2}. Define
    $$
    \alpha(z):=\frac{A'(z)}{A(z)}=\frac{2gg'+g''}{g^2+g'}.
    $$
Then we may write
    \begin{equation}\label{poly-eqn}
    \left(2g'/g-\alpha(z)\right)g^2+\left(g''/g-\alpha(z)g'/g\right)g=0.
    \end{equation}
If
    $$
    2g'/g-\alpha(z)=\frac{2(g')^2-gg''}{g(g^2+g')}
    $$
vanishes identically, then $g''/g'=2g'/g$. After integration, this transforms into a Riccati equation
$g'=C_1g^2$ for some constant $C_1\neq 0$. Substituting $g=1/u$ into this equation, we get $u'=-C_1$, and finally
$g(z)=(C_2-C_1z)^{-1}$, where $C_2\in\C$. Now $f$ is either a polynomial or non-entire. Since both cases are impossible, we must have
$2g'/g-\alpha(z)\not\equiv 0$. On the other hand, if $g''/g-\alpha(z)g'/g\equiv 0$, then once again $g''/g'=2g'/g$, which is impossible.
Hence \eqref{poly-eqn} together with the first main theorem yields
     \begin{equation}\label{mrg}
    \begin{split}
    m(r,g) &\leq m\left(r,\frac{1}{2g'/g-\alpha}\right)+S(r,g)\\
    &\leq T\left(r,2\frac{g'}{g}-\alpha\right)+S(r,g)\\
    &= N\left(r,2\frac{g'}{g}-\frac{A'}{A}\right)+S(r,g).
    \end{split}
    \end{equation}
All poles of the function $G=2g'/g-A'/A$ must be simple since $G$ can be expressed as the logarithmic derivative of $g^2/A$.
Since $f$ is entire, all poles of $G$ must occur at the zeros of $ff'A$. The assertion now follows from \eqref{mrg}.
\end{proof}

We take this opportunity to state an alternative formulation of Lemma~\ref{FH-lem}, which is particularly applicable
in the case when $A(z)$ is an exponential polynomial of the form \eqref{exp-poly2}.

\begin{lemma}\label{FH-lem2}
Let $f$ be a nontrivial solution of \eqref{lde2}, where $A(z)=B(z)+C(z)$ is such that $A(z)$ and $B(z)$ are transcendental entire
and $C(z)\not\equiv 0$ is entire. Then
    \begin{equation}\label{FH-estim3}
    m\left(r,\frac{f'}{f}\right)\leq \overline{N}\left(r,\frac{1}{ff'B}\right)+m(r,C)+S(r,C)+S\left(r,\frac{f'}{f}\right),
    \end{equation}
unless $B(z)$ solves $B'-(2g'/g)B=0$ or $B'-(g''/g')B=0$, where $g=f'/f$.
\end{lemma}

\begin{proof}
Let $g=f'/f$. Now $-B(z)=g^2+g'+C(z)$ yields
    $$
    \frac{B'(z)}{B(z)}=\frac{2gg'+g''+C'(z)}{g^2+g'+C(z)},
    $$
which we denote by $\beta(z)$ for short. Then we may write
    \begin{equation}\label{poly-eqn2}
    \left(2g'/g-\beta(z)\right)g^2+\left(g''/g-\beta(z)g'/g\right)g+C'(z)-\beta(z)C(z)=0.
    \end{equation}
If $f$ has a zero, then $g$ has a pole, and no entire function $B$ can solve either of $B'-(2g'/g)B=0$ or $B'-(g''/g')B=0$.
Since these two equations are excluded in any case, the coefficients of $g^2$ and of $g$ in \eqref{poly-eqn2} do
not vanish identically. Hence the rest of the proof of \eqref{FH-estim3} follows that of Lemma~\ref{FH-lem}.
\end{proof}

To analyze Lemma~\ref{FH-lem2}, suppose that $C(z)$ is a small function in the sense that $m(r,C)=S(r,A)$.
This happens, for example, if $A(z)$ is given by \eqref{exp-poly2}, where $H_0(z)=C(z)$. In general, we have
    \begin{equation}\label{Ag}
    T(r,A)=m(r,A)=m\left(r,\left(g+\frac{g'}{g}\right)g\right) \leq 2m(r,g)+S(r,g).
    \end{equation}
Therefore \eqref{FH-estim3} reduces to
    \begin{equation}\label{FH-estim4}
     m\left(r,\frac{f'}{f}\right)\leq \overline{N}\left(r,\frac{1}{ff'B}\right)+S\left(r,\frac{f'}{f}\right).
    \end{equation}
The advantage of this estimate as opposed to \eqref{FH-estim2} is clear when $f$ is zero-free: The functions $f'$ and $B$ may share
the same zeros while $f'$ and $A$ don't.

\begin{example}\label{half-ex}
If $A(z)=-\frac14 (e^{2z}-2e^z+4)$, then $f(z)=\exp\left(\frac12 e^z-z\right)$
is a zero-free solution of \eqref{lde2}. We may choose $B(z)=-\frac14e^z(e^z-2)$ and $C(z)=-1$ for
the representation $A(z)=B(z)+C(z)$. Since $f'=\frac12(e^z-2)f$, we see that $f'$ and $B(z)$ share the same zeros. In particular,
the equality in \eqref{FH-estim4} (as well as in \eqref{FH-estim3}) holds. Since $A(z)=-\frac14(e^z-w)(e^z-\overline{w})$, where $w=1+\sqrt{3}i$,
the zeros of $A(z)$ are different from those of $f'$.
\end{example}

\begin{example}
In the previous example, we have $g=f'/f=\frac12(e^z-2)$. The two special cases $B'/B=2g'/g$ and $B'/B=g''/g'$ appearing in Lemma~\ref{FH-lem2}
can be obtained by choosing $B(z)=-g^2$ and $C(z)=-g'$, or $B(z)=-g'$ and $C(z)=-g^2$, respectively.
\end{example}

\begin{remark}\label{g2-rem}
Suppose that a non-trivial solution $f$ of \eqref{lde2} has no zeros, and that $g=f'/f$ is an exponential polynomial.
Then $g^2+g'$ is also an exponential polynomial, and the convex hull of its leading coefficients matches with that of $g^2$. Therefore
    $$
    2m(r,g)=T(r,g^2)=T(r,g^2+g')+S(r,g)=T(r,A)+S(r,g).
    $$
Since $A(z)$ is entire, the second main theorem \cite[Theorem 2.5]{Hayman} yields
    \begin{equation}\label{g-exppoly}
    m\left(r,\frac{f'}{f}\right)\leq 2^{-1}\overline{N}\left(r,\frac{1}{A-a_1}\right)+2^{-1}\overline{N}\left(r,\frac{1}{A-a_2}\right)+S\left(r,\frac{f'}{f}\right)
    \end{equation}
for any two distinct small target functions $a_1(z),a_2(z)$ of $A(z)$.
\end{remark}

\begin{example}\label{power2-ex}
If $A(z)=-\frac{1}{16}(e^z+1)^2$, then $f(z)=\exp\left(\frac14(e^z-z)\right)$ solves \eqref{lde2}. Choosing
$a_1=0$ and $a_2=-1/16$, we deduce by Theorem~\ref{Stein-thm} that
    \begin{eqnarray*}
    m\left(r,\frac{f'}{f}\right)&=&m(r,(e^z-1)/4)=\frac{r}{\pi}+O(1)\\
    2\overline{N}\left(r,\frac{1}{A}\right)&=&{N}\left(r,\frac{1}{A}\right)=\frac{2r}{\pi}+o(r)\\
    \overline{N}\left(r,\frac{1}{A+1/16}\right)&=&{N}\left(r,\frac{1}{A+1/16}\right)={N}\left(r,\frac{1}{e^z+2}\right)=\frac{r}{\pi}+o(r).
    \end{eqnarray*}
Hence the equality in \eqref{g-exppoly} holds.
\end{example}

\begin{remark}\label{PeQ-rem}
Suppose that $f=Pe^Q$ is a non-trivial solution of \eqref{lde2}, where $P,Q$ are exponential polynomials. We may suppose that
$P$ has at least two terms (in which case $P$ has infinitely many zeros), for otherwise the situation reduces essentially to the one in Remark~\ref{g2-rem}. Now
    $$
    m\left(r,\frac{f'}{f}\right)=m\left(r,\frac{P'}{P}+Q'\right)=m\left(r,Q'\right)+O(\log r).
    $$
If $g=f'/f$, we observe that
    $$
    -A=g^2+g'=(Q')^2+\left(2\frac{P'}{P}+\frac{Q''}{Q'}\right)Q'+\frac{P''}{P}.
    $$
An easy modification of \cite[Theorem~1.12]{YY} shows that
	$$
	m(r,A)=2m(r,Q')+O(\log r),
	$$
and hence
    \begin{equation}\label{g-exppoly2}
    m\left(r,\frac{f'}{f}\right)\leq 2^{-1}m(r,A)+O(\log r).
    \end{equation}
Moreover, \eqref{g-exppoly} holds again.
\end{remark}

\begin{example}\label{rationalnumbers-ex}
For $w_1=-\frac{72}{49}+\frac{6}{49}\sqrt{95}$ and $w_2=-\frac{72}{49}-\frac{6}{49}\sqrt{95}$, let
    $$
    A(z)=-\frac{49}{4}-36e^{-z}-9e^{-2z}=-\frac{49}{4}e^{-2z}(e^z-w_1)(e^z-w_2).
    $$
Then \eqref{lde2} has a solution
$f(z)=(e^z+1)\left(e^z+\frac12\right)\exp\left(3e^{-z}-\frac{11}{2}z\right)$, for which
    $$
    \frac{f'(z)}{f(z)}=-\frac{14e^{2z}+39e^{z}+29+6e^{-z}}{4(e^z+1)\left(e^z+\frac12\right)}.
    $$
Thus \cite[Satz~1]{Stein2} gives us
    $$
    m\left(r,\frac{f'}{f}\right)=\frac{r}{\pi}+o(r).
    $$
We deduce by Theorem~\ref{Stein-thm} that
    \begin{eqnarray*}
    m\left(r,\frac{1}{A}\right) &=& O(\log r)\\
    \overline{N}\left(r,\frac{1}{A}\right)&=&{N}\left(r,\frac{1}{A}\right)=\frac{2r}{\pi}+o(r)\\
    \overline{N}\left(r,\frac{1}{A+49/4}\right)&=&{N}\left(r,\frac{1}{e^{-z}+4}\right)=\frac{r}{\pi}+o(r),
    \end{eqnarray*}
and hence the equality in \eqref{g-exppoly2} holds while \eqref{g-exppoly} is strict.
\end{example}

\begin{example}
If $A(z)=-\frac14e^{-2z}(e^{4z}-4e^{3z}+3e^{2z}-4e^{z}+1)$, then \eqref{lde2} has a solution
$f(z)=(e^{z}+1)\exp\left(-\frac12 e^{-z}-\frac12 e^{z}-\frac{z}{2}\right)$, for which
$\frac{f'(z)}{f(z)}=\frac{e^{-z}-e^{2z}}{2(e^z+1)}$.
Now
	\begin{eqnarray*}
	m\left(r,\frac{f'}{f}\right) &=& \frac{2r}{\pi}+o(r),\\
	\overline{N}\left(r,\frac{1}{A}\right) &=& N\left(r,\frac{1}{A}\right)=m(r,A)=\frac{4r}{\pi}+o(r),
	\end{eqnarray*}
and hence the equality in \eqref{g-exppoly2} holds while \eqref{g-exppoly} is strict for
the choice of $a_1=0$ and $a_2\in\C\setminus\{0\}$.
\end{example}

\section{A variant of Theorem~\ref{K-thm-IT}}

We prove a variant of Theorem~\ref{K-thm-IT} according to which every nontrivial solution of \eqref{lde2}
has either a lot of zeros or a lot of critical points. The proof relies on Frank-Hennekemper type of estimates,
which were discussed in the previous section.
The assertion involves a general transcendental entire coefficient $A(z)$ of finite order.
We will conclude this section with remarks on the case of exponential polynomial
coefficient.

\begin{theorem}\label{K-thm}
Let $A(z)$ be a transcendental entire function of order $\rho(A)$, and let $f$ be a nontrivial solution of \eqref{lde2}.
\begin{itemize}
\item[{\rm (1)}] If $\overline{N}\left(r,\frac{1}{A}\right)=S(r,A)$, then
                \begin{equation}\label{a}
                \limsup_{r\to\infty} \frac{\overline{N}\left(r,\frac{1}{f}\right)}{T\left(r,A\right)}
	          \geq\frac{1}{8}.
                \end{equation}
                In particular, $\overline{\lambda}(f)\geq \rho(A)$.
\item[{\rm (2)}] If $\overline{N}\left(r,\frac{1}{A}\right)\neq S(r,A)$ but \eqref{K} holds for $K>2$, then
                \begin{equation}\label{b}
                \limsup_{r\to\infty} \frac{\overline{N}\left(r,\frac{1}{ff'}\right)}{\overline{N}\left(r,\frac{1}{A}\right)}\geq \frac{K-2}{2}>0.
                \end{equation}
                In particular, $\max\{\overline{\lambda}(f),\overline{\lambda}(f')\}
		   \geq\rho(A)$.
\end{itemize}
\end{theorem}

For the proof we need the next lemma,
which is elementary but may be of independent interest in some other contexts also. Recall first that
the upper (linear) density of a set $F\subset (0,\infty)$ is defined by
	$$
	\overline{\operatorname{dens}}(F)=\limsup_{r\to\infty}\frac{1}{r}\int_{F\cap [0,r]}dt.
	$$

\begin{lemma}\label{comparable-growth-lem}
Let $f$ and $g$ be entire functions. If $\rho(f)<\rho(g)$, then there exists a set $F\subset (0,\infty)$ such
that $\overline{\operatorname{dens}}(F)=1$ and
    \begin{equation}\label{F}
    T(r,f)=o(T(r,g)),\quad r\in F.
    \end{equation}
Analogous assertions $N(r,1/f)=o(N(r,1/g))$ and $N(r,1/f)=o(T(r,g))$ hold for $r\in F$ in the
cases $\lambda(f)<\lambda(g)$ and $\lambda(f)<\rho(g)$, respectively.
\end{lemma}

\begin{proof}
Choose $\veps\in \left(0,\frac{\rho(g)-\rho(f)}{3}\right)$ if $\rho(g)<\infty$, or $\veps\in (0,1)$ if $\rho(g)=\infty$.
Then choose $\delta\in \left(0,\min\{\veps/2\rho(f),1/2\}\right)$. By the assumptions, there exists a sequence $\{r_n\}$
of positive numbers tending to infinity such that $T(r_n,g)\geq r_n^{\rho(f)+3\veps}$. Set $s_n=r_n^{1+\delta}$. Then
for $r_n\leq r\leq s_n$ and $n$ large enough,
    \begin{eqnarray*}
    T(r,f)&\leq& r^{\rho(f)+\veps}\leq s_n^{\rho(f)+\veps}=r_n^{(1+\delta)(\rho(f)+\veps)}\\
    &\leq& r_n^{\rho(f)+2\veps}= o(T(r_n,g))=o(T(r,g)).
    \end{eqnarray*}
It is easy to see that the set $F=\cup_n[r_n,s_n]$ satisfies $\overline{\operatorname{dens}}(F)=1$,
so that \eqref{F} is proved. The analogous assertions are proved similarly.
\end{proof}

\bigskip
\noindent
\emph{Proof of Theorem~\ref{K-thm}.}
(1)  Let $f$ be a non-trivial solution of \eqref{lde2}, where the coefficient $A(z)$ satisfies
$\overline{N}\left(r,\frac{1}{A}\right)= S(r,A)$, and suppose on the contrary to
the assertion \eqref{a} that
	\begin{equation}\label{contrary-a}
       \limsup_{r\to\infty} \frac{\overline{N}\left(r,\frac{1}{f}\right)}{T\left(r,A\right)}
	 <\frac{1}{8}.
       \end{equation}
Denoting $g=f'/f$, we deduce from \eqref{reads-as} and \eqref{Ag} that
	\begin{equation}\label{4g}
	m(r,g)\leq 4\overline{N}\left(r,\frac{1}{f}\right)+S(r,g).
	\end{equation}
Using \eqref{contrary-a} and \eqref{Ag} in \eqref{4g} results in a contradiction.

Suppose on the contrary to the remaining assertion in Part (1) that
    $
    \overline{\lambda}(f)<\rho(A).
    $
Then Lemma~\ref{comparable-growth-lem} and \eqref{Ag} yield
    $$
    \overline{N}\left(r,\frac{1}{f}\right)=o(T(r,A))=o(T(r,g)),\quad r\in F,
    $$
where $\overline{\operatorname{dens}}(F)=1$. Again \eqref{4g} results in a contradiction.

(2) Let $f$ be a non-trivial solution of \eqref{lde2}, where the coefficient $A(z)$ satisfies
$\overline{N}\left(r,\frac{1}{A}\right)\neq S(r,A)$, yet \eqref{K} holds for $K>2$. Suppose on the contrary to the assertion \eqref{b} that
	\begin{equation}\label{contrary-b}
    	\limsup_{r\to\infty} \frac{\overline{N}\left(r,\frac{1}{ff'}\right)}{\overline{N}
	\left(r,\frac{1}{A}\right)}<\frac{K-2}{2}.
    	\end{equation}
We conclude by Lemma~\ref{FH-lem} and \eqref{contrary-b} that
    	\begin{equation}\label{g}
	\begin{split}
    	m(r,g) &\leq \overline{N}\left(r,\frac{1}{ff'}\right)
	+\overline{N}\left(r,\frac{1}{A}\right)+S(r,g)\\
	&\leq(M+1)\overline{N}\left(r,\frac{1}{A}\right)+S(r,g),\quad r\geq R,
    	\end{split}
	\end{equation}
where $M<(K-2)/2$ and $R>0$ are constants. Using \eqref{K} and \eqref{Ag} in \eqref{g}
results in a contradiction.
The remaining assertion $\max\{\overline{\lambda}(f),\overline{\lambda}(f')\}
\geq\rho(A)$ in Part (2) follows by Lemma~\ref{comparable-growth-lem}. \hfill$\Box$

\begin{remark}
If $A(z)$ is an exponential polynomial of the normalized form \eqref{exp-poly2}, then
the condition $\overline{N}\left(r,\frac{1}{A}\right)=S(r,A)$ in Theorem~\ref{K-thm}(1)
is possible only in the special case when $m=1$ and $H_0(z)\equiv 0$. Moreover,
\eqref{K} holds if $C(\co(W_0))>KC(\co (W))$ and $H_0(z)\equiv 0$. However,
examples in Section~\ref{FH-sec} show that  $\max\{\overline{\lambda}(f),\overline{\lambda}(f')\}\geq\rho(A)$ may hold if $A(z)$ is an exponential polynomial
which does not satisfy these requirements. In such cases the use of \eqref{FH-estim4}
instead of \eqref{FH-estim3} may result in better estimates, as in Example~\ref{half-ex}.
\end{remark}

\section{A generalization of Theorem~\ref{Ishizaki-Tohge}(c)}

The next result is a slight generalization of Part (c) in Theorem~\ref{Ishizaki-Tohge}.

\begin{theorem}\label{n-thm-c}
Let $A(z)$ be an exponential polynomial of the normalized form
    \begin{equation}\label{A0}
    A(z)=H_0(z)+H_1(z)e^{\zeta_1z^n}+\cdots +H_m(z)e^{\zeta_mz^n},\quad m\geq 2,
    \end{equation}
where the functions $H_j(z)$ are either exponential polynomials of order $<n$ or ordinary polynomials in $z$.
Denote $B(z)=A(z)-H_1(z)e^{\zeta_1z^n}$. Suppose that $h_A(\theta)>2h_B(\theta)$ whenever $h_B(\theta)>0$, and
that $C(\co(W_0^A))>2C(\co(W_0^B))$.
Then $\lambda(f)\geq n$ for any nontrivial solution $f$ of \eqref{lde2}.
\end{theorem}

We discuss the necessity of the assumptions in Theorem~\ref{n-thm-c} as follows.

\begin{example}\label{8}
Let $f(z)=\exp\left(e^{z^n}\right)$, and denote $C(z)=f'(z)/f(z)=nz^{n-1}e^{z^n}$. Then $f$ is a zero-free solution
of \eqref{lde2}, where
    \begin{eqnarray*}
    A(z)&=&-C'(z)-C(z)^2\\
    &=&-\left(n(n-1)z^{n-2}+n^2z^{2(n-1)}\right)e^{z^n}-n^2z^{2(n-1)}e^{2z^n}.
    \end{eqnarray*}
Defining $B(z)=A(z)+C(z)^2=-C'(z)$, we have $h_A(\theta)=2h_B(\theta)$ whenever $h_B(\theta)>0$,
and $C(\co(W_0^A))=2C(\co(W_0^B))=4$.
\end{example}

The proof of Theorem~\ref{n-thm-c} relies partially on the ideas used in \cite{Ishizaki, Ishizaki-Tohge},
but is also based on Steinmetz' treatment of exponential polynomials as well as on the following lemma.

\begin{lemma}\label{apu-lemma}
Let $A(z)$ be an exponential polynomial of the form \eqref{A0}.
Let $f=\pi e^g$ be a solution of \eqref{lde2} such that $\lambda(f)=\rho(\pi)<\infty$.
Then $\rho(A)=\rho(g)$, and the following assertions hold.\\[-20pt]
\begin{itemize}
\item[{\rm (1)}] $T(r,A)\leq 2T(r,g')+O(\log r)$ and $T(r,g')\leq T(r,A)+O(\log r)$. \\[-20pt]
\item[{\rm (2)}] If $h_A(\theta)>0$, then $h_A(\theta)=2h_{g'}(\theta)$.\\[-20pt]
\item[{\rm (3)}] If $h_A(\theta)\leq 0$, then $\log^+ |g'(re^{i\theta})|\leq O\left(r^{n-1}\right)+O(\log r)$
                 as $r\to\infty$ for almost every such $\theta$. Conversely, if $h_{g'}(\theta)\leq 0$, then
                 $\log^+ |A(re^{i\theta})|\leq o\left(r^{n}\right)$ as $r\to\infty$ for almost every such $\theta$.
\end{itemize}
\end{lemma}

The assertion $\rho(A)=\rho(g)$ in Lemma~\ref{apu-lemma} is known -- the existing proofs are typically based on
Clunie's theorem. We will give an alternative short proof which can be modified to justify the assertions
in Parts (1)--(3). Remark~\ref{PeQ-rem} above shows that if $\pi(z)$ and $g(z)$
are exponential polynomials, then the first inequality in Part (1) is in fact an equality.

\bigskip
\noindent
\emph{Proof of Lemma~\ref{apu-lemma}.}
A substitution of $f=\pi e^g$ to \eqref{lde2} gives
    \begin{equation}\label{subst}
    \pi''+2g'\pi'+(g''+(g')^2+A)\pi=0,
    \end{equation}
where $\pi$ is a canonical product formed by the zeros of $f$.
By re-writing \eqref{subst} as
    \begin{equation}\label{AA}
    A=-g''-(g')^2-2g'\frac{\pi'}{\pi}-\frac{\pi''}{\pi},
    \end{equation}
we see by means of a standard lemma on the logarithmic derivative that $\rho(A)\leq \max\{\rho(g'),\rho(g'')\}=\rho(g)$.
Next, we write \eqref{subst} as
    \begin{equation}\label{g-prime}
    (g')^2=-A-g''-2g'\frac{\pi'}{\pi}-\frac{\pi''}{\pi}.
    \end{equation}
Dividing \eqref{g-prime} by $g'$, and considering separately the cases $|g'(z)|\leq 1$ and
$|g'(z)|>1$, it follows that
    \begin{equation}\label{g-prime2}
    \begin{split}
    \log^+|g'(z)| \leq &\log^+ |A(z)|+\log^+\left|\frac{g''(z)}{g'(z)}\right|\\
    &+\log^+\left|\frac{\pi'(z)}{\pi(z)}\right|+\log^+\left|\frac{\pi''(z)}{\pi(z)}\right|+O(1)
    \end{split}
    \end{equation}
holds for every $z\in\C$. Therefore $\rho(g)=\rho(g')\leq \rho(A)$.

(1) Equation \eqref{AA} yields
    \begin{eqnarray*}
    T(r,A) &=& m(r,A)=m\left(r,g'\left(\frac{g''}{g'}+g'+2\frac{\pi'}{\pi}\right)\right)+O(\log r)\\
    &\leq & 2m(r,g')+O(\log r)=2T(r,g')+O(\log r).
    \end{eqnarray*}
The second assertion $T(r,g')\leq T(r,A)+O(\log r)$ follows similarly from \eqref{g-prime}.

(2) By applying \cite[Corollary~1]{Gundersen} to \eqref{AA}, we have
    $$
    h_A(\theta)\leq \max\{0,h_{g''}(\theta),h_{(g')^2}(\theta),h_{g'}(\theta)\}
    $$
for almost all $\theta$. Using \eqref{derivative-indicator} and the continuity of indicator functions,
we have $h_A(\theta)\leq \max\{0,h_{(g')^2}(\theta),h_{g'}(\theta)\}$ for all $\theta$. If $h_A(\theta)>0$,
then it follows that $h_A(\theta)\leq 2h_{g'}(\theta)$. Similarly, by applying
\cite[Corollary~1]{Gundersen} to \eqref{g-prime}, we get $2h_{g'}(\theta)\leq h_A(\theta)$
for all $\theta$ for which $h_A(\theta)>0$.

(3) Let $\theta$ be such that $h_A(\theta)\leq 0$. Since $A(z)$ is an exponential polynomial, we conclude that
$\log^+ |A(re^{i\theta})|\leq O\left(r^{n-1}\right)+O(\log r)$. The assertion for $g'$ follows from \eqref{g-prime2}
by means of \cite[Corollary~1]{Gundersen}. Conversely, if $h_{g'}(\theta)\leq 0$, then
$\log^+ |g'(re^{i\theta})|\leq o\left(r^{n}\right)$ as $r\to\infty$ by the definition of the indicator.
The assertion for $A(z)$ follows from \eqref{AA}.\hfill$\Box$

\bigskip
If $H_0(z)\not\equiv 0$ in \eqref{A0}, then $h_A(\theta)\geq 0$ for every $\theta$. It is possible that $h_A(\theta)=0$
on an interval or on finitely many subintervals of $[-\pi,\pi)$. A trivial example would be $A(z)=e^z$.
The leading coefficients can also be chosen such that $h_A(\theta)=0$ only on finitely many points on $[-\pi,\pi)$,
and $h_A(\theta)>0$ on the rest of the interval. For example, $A(z)=e^z+e^{-z}$ has this property.
In such a case $h_A(\theta)=2h_{g'}(\theta)$ holds for every $\theta$ by the continuity of (bounded) indicator functions.

\begin{example}
If
    \begin{equation*}
    \begin{split}
    A(z)=\ & e^{iz}+e^{-iz}-e^{z}-e^{-z}+e^{2iz}+e^{-2iz}-e^{2z}-e^{-2z}\\
    &\!+2ie^{(1-i)z}+2ie^{(i-1)z}-2ie^{(1+i)z}-2ie^{-(1+i)z},
    \end{split}
    \end{equation*}
then \eqref{lde2} possesses a zero-free solution $f(z)=\exp\left(e^z+e^{-z}+e^{iz}+e^{-iz}\right)$.
The convex hull of the conjugates of the leading coefficients of $A(z)$ is a square determined by the corners $\pm2,\pm2i$ and
has a circumference $8\sqrt{2}$. In particular, \eqref{lde2} may have zero-free solutions even if $h_A(\theta)>0$ for every $\theta$.
\end{example}

\bigskip
\noindent
\emph{Proof of Theorem~\ref{n-thm-c}.}
Suppose on the contrary to the assertion that $f=\pi e^g$ with $\lambda(\pi)<n$ solves \eqref{lde2}. We may
choose a constant $\alpha>0$ such that $\max\{\lambda(\pi),n-1\}<\alpha<n$. Substituting $f$ in \eqref{lde2} gives us
\eqref{g-prime}. Then eliminating $e^{\zeta_1z^n}$ from \eqref{g-prime} and writing $R=H_1'/H_1+n\zeta_1z^{n-1}$,
we have
    \begin{equation}\label{elimination}
    2Ug'=C+D,
    \end{equation}
where
    \begin{eqnarray*}
    U&=&g''-Rg'/2,\\
    C&=&\sum_{j=2}^m (RH_j-H_j'-n\zeta_jz^{n-1}H_j)\exp(\zeta_jz^n),\\
    D&=& -g'''+\left(R-2\frac{\pi'}{\pi}\right)g''+2\left(R\frac{\pi'}{\pi}-\left(\frac{\pi'}{\pi}\right)'\right)g'
    +R\frac{\pi''}{\pi}-\left(\frac{\pi''}{\pi}\right)'.
    \end{eqnarray*}
If $RH_j-H_j'-n\zeta_jz^{n-1}H_j=0$ for some index $j$, then $H_j$ solves
    $$
    H_j'+\left(n(\zeta_j-\zeta_1)z^{n-1}-\frac{H_1'(z)}{H_1(z)}\right)H_j=0,
    $$
that is, $H_j(z)=KH_1(z)\exp((\zeta_1-\zeta_j)z^n)$ for some constant $K\in\C$. But this is
a contradiction since $\zeta_j\neq \zeta_1$, $H_j\not\equiv 0$ and $\rho(H_j)\leq n-1$ for all $j$.

We may write \eqref{elimination} in the alternative form
    \begin{equation}\label{elimination2}
    F_1g'=F_2,
    \end{equation}
where
    \begin{eqnarray}
    F_1 &=& 2U+\frac12 R'-\frac14 R^2-R\frac{\pi'}{\pi}
    +2\frac{\pi''}{\pi}-2\left(\frac{\pi'}{\pi}\right)^2,\nonumber \\
    F_2 &=& -U'+\frac12 RU-2\frac{\pi'}{\pi}U+R\frac{\pi''}{\pi}-\frac{\pi'''}{\pi}
    +\frac{\pi''\pi'}{\pi^2}+C.\label{C2}
    \end{eqnarray}

Using the representation \eqref{elimination}, we proceed to show that $F_1\equiv 0$ and $F_2\equiv 0$.
A key step is to prove that
    \begin{equation}\label{TU}
    T(r,U)=O(r^\alpha).
    \end{equation}
The possible poles of $U$ are among the zeros of $H_1$, thus $V=UH_1$ is an entire function.
Equation \eqref{elimination} can now be written as
    \begin{equation}\label{elimination3}
    2Vg'=(C+D)H_1.
    \end{equation}
As a passage to \eqref{TU}, we prove
    \begin{equation}\label{U}
    \log^+ |V(re^{i\theta})|= O(r^\alpha),\quad r\to\infty,\ \theta\not\in E_0,
    \end{equation}
where $E_0\subset [-\pi,\pi)$ is of measure zero.
If $H_1$ is a polynomial, then $R$ is a rational function, and so $\log^+ |R(z)|=O(\log |z|)$.
Hence we suppose that $m=\rho(H_1)\geq 1$, and keep in mind that $m\leq n-1<\alpha$. By making use of
\eqref{f-behavior} and \eqref{derivative-indicator} for $H_1$ instead of $f$, we deduce that
    \begin{equation}\label{R}
    \log^+ |R(re^{i\theta})|=o\left(r^{m}\right),\quad r\to\infty,
    \end{equation}
for every $\theta\in [-\pi,\pi)$ with at most finitely many exceptions.
If $z$ is such that $|g'(z)|\leq 1$, then $U=(g''/g'-R/2)\cdot g'$ together with \eqref{R} and \cite[Corollary~1]{Gundersen} yield
    $$
    \log^+ |V(z)|\leq \log^+ \left(\left|\frac{g''(z)}{g'(z)}\right|+\frac{1}{2}|R(z)|\right)+\log^+|H_1(z)|
    \leq O\left(|z|^\alpha\right)
    $$
as $|z|\to\infty$ with $\arg(z)\not\in E_0$, where $\operatorname{meas}\,(E_0)=0$. Hence from now on we may suppose that $|g'(z)|>1$. If
$\theta\not\in E_0$ is such that $h_C(\theta)\leq 0$, then
    $$
    \log^+ |C(re^{i\theta})|\leq O\left(r^{n-1}\right),
    $$
and dividing \eqref{elimination3} by $2g'$ gives us
\eqref{U}. On the other hand, if $\theta\not\in E_0$ is such that $h_C(\theta)>0$, then
$h_B(\theta)=h_C(\theta)>0$, since we have proved that none of the coefficients of the exponential polynomial $C$ vanishes.
We have by the assumption and by Lemma~\ref{apu-lemma} that
    $$
    2h_{g'}(\theta)=h_A(\theta)>2h_C(\theta).
    $$
Dividing \eqref{elimination3} again by $2g'$ gives us \eqref{U} in this case also.

Since $V$ is entire and of finite order, we may use the standard Phragm\'en-Lindel\"of principle to deduce that the
estimate in \eqref{U} is uniform, and that the exceptional set $E_0$ can be ignored. Therefore $m(r,V)=O\left(r^\alpha\right)$, and so
    $$
    m(r,U)=m(r,V/H_1)\leq m(r,V)+T(r,H_1)+O(1)= O\left(r^\alpha\right).
    $$
Finally, since $\lambda(1/U)\leq n-1$, the assertion \eqref{TU} follows.

Next, we continue the proof by estimating $T(r,F_1)$ and $T(r,F_2)$ under the assumption that $F_1F_2\not\equiv 0$.
Since $\rho(R)\leq n-1$ and $\rho(\pi)<\alpha$, we obtain by \eqref{TU} that
    $$
    T(r,F_1)=O(r^\alpha).
    $$
In the same manner we deduce that
    $$
    T(r,F_2)\leq T(r,C)+O(r^\alpha).
    $$
The identity $T(r,C)=T(r,B)+O(r^\alpha)$
being clear by Theorem~\ref{Stein-thm}, we then get from \eqref{elimination2} that
    $$
    T(r,g')\leq T(r,F_1)+T(r,F_2)+O(1)\leq T(r,B)+O(r^\alpha).
    $$
By using the assumption $C(\co(W_0^A))>2C(\co(W_0^B))$ together with Lemma~\ref{apu-lemma} and Theorem~\ref{Stein-thm},
we arrive at a contradiction. Hence at least one of $F_1$ or $F_2$ must vanish. A fortiori they both
vanish by \eqref{elimination2}.

Finally, since $F_2\equiv 0$, we infer from \eqref{C2} that
    $$
    T(r,C)=O(r^\alpha),
    $$
which is clearly a contradiction. Thus $\lambda(f)\geq n$. \hfill$\Box$

\section{An improvement of the $\frac{1}{16}$-theorem}

The following improvement of the $\frac{1}{16}$-theorem (see Theorem~\ref{BLL-thm}) is motivated
by the assumptions in Theorem~\ref{n-thm-c}.

\begin{theorem}\label{referee-thm}
Suppose that $A(z)=T(z)+B(z)$, where
    $$
    T(z)=H(z)e^{\zeta z^n},\quad \zeta\in\C\setminus\{0\},
    $$
$H(z)$ is an entire function of order $<\alpha<n$, and $B(z)$
is a finite order entire function with the following property: For each $\theta\in [-\pi,\pi)$
there exists $s(\theta)<1/2$ with
    \begin{equation}\label{B}
    \log^+|B(re^{i\theta})|\leq s(\theta)\max\left\{\Re (\zeta e^{in\theta}),0\right\}r^n+r^\alpha,
    \end{equation}
as $r\to\infty$, but not necessarily uniformly in $\theta$. If \eqref{lde2} possesses a nontrivial
solution $f$ with $\lambda(f)<\alpha<n$, then $f$ and $H(z)$ have no zeros, and
    \begin{equation}\label{B2}
    B(z)=-\frac{1}{16}\left(\frac{T'(z)}{T(z)}\right)^2\\
    +\frac14 \left(\frac{T'(z)}{T(z)}\right)'.
    \end{equation}
Moreover, \eqref{lde2} admits in this case a zero-free solution base.
\end{theorem}

It is easy to see that \eqref{B2} forces $B(z)$ to be a polynomial of degree $2(n-1)$ with
leading term $-16^{-1}n^2\zeta^2z^{2(n-1)}$, so the situation is not that much different from
that in Theorem~\ref{BLL-thm}. As for the sharpness of Theorem~\ref{referee-thm},
if we permit $s(\theta)$ to attain the value $1/2$, then the existence of a zero-free solution
does not imply \eqref{B2}, see Example~\ref{8}.

\begin{lemma}\label{minmod-lemma}
Let $g$ be an entire function of order $\rho(g)<\alpha$. Then $|\log |g(z)||\leq |z|^\alpha$ for
all $z$ outside of an $R$-set.
\end{lemma}

\begin{proof}
The inequalities
    $
    \log |g(z)|\leq \log M(|z|,g)\leq |z|^\alpha
    $
being trivial, it suffices to prove $\log |g(z)| \geq -|z|^\alpha$ for suitable $z$.
We may write $g=Pe^h$, where $h$ is a polynomial of degree $<\alpha$, and $P$ is a canonical
product formed with the zeros $z_n$ of $g$. If $g$ has no zeros or has at most finitely many zeros (that is,
$P$ is a polynomial), the assertion is trivial. If $P$ is transcendental, then
\cite[Theorem V.~19]{Tsuji} gives us
    $$
    \log^+\frac{1}{|P(z)|}=O\left(|z|^{\rho(P)+\veps}\right)
    $$
for all $z$ outside of an $R$-set \cite[p.~84]{Laine}, where we have fixed $\veps<\alpha-\rho(P)$. This yields the assertion.
\end{proof}

\noindent
\emph{Proof of Theorem~\ref{referee-thm}.}
First we observe that the assumption \eqref{B} together with the Phragm\'en-Lindel\"of principle yield
$T(r,B)=O\left(r^n\right)$. Moreover, the assumption \eqref{B} holds with $B$ being replaced by $B'$
for almost all $\theta$.

Set $F(z)=f'(z)/f(z)$. Then $\lambda(1/F)< \alpha$
by the assumption, and so $\rho(F)<\alpha$ by the standard lemma on the logarithmic derivative.
By \cite[Proposition~5.12]{Laine} there exists a constant $N_0>0$ and an $R$-set $E_1$ surrounding the
zeros of $f$ such that
    \begin{equation}\label{e1}
    \left|\frac{F''(z)}{F(z)}\right|+\left|\frac{F'(z)}{F(z)}\right|+\left|\frac{T'(z)}{T(z)}\right|
    =O\left(|z|^{N_0}\right),\quad z\not\in E_1.
    \end{equation}
By Lemma~\ref{minmod-lemma} there exists an $R$-set $E_2$ such that
    \begin{equation}\label{e2}
    |\log |H(z)||\leq |z|^\alpha,\quad z\not\in E_2.
    \end{equation}
Recall that the set of angles $\theta$ for which the ray $re^{i\theta}$ meets infinitely many discs
of a given $R$-set has linear measure zero by \cite[Lemma~5.9]{Laine}. Thus the estimates in \eqref{e1} and \eqref{e2} are valid
for almost all $\theta$ as $|z|\to\infty$.

We have the identities
    \begin{equation}\label{e3}
    F'+F^2 = -A=-T-B,
	\end{equation}
	\begin{equation}\label{e4}
    F''+2FF' = -T'-B'=t(F'+F^2+B)-B',
    \end{equation}
where $t=T'/T$ has order $\rho(t)<\alpha$. Set $G=2F'-tF$. Then \eqref{e4} yields
    \begin{eqnarray*}
    FG &=& 2FF'-tF^2=-F''+t(F'+B)-B'.
    \end{eqnarray*}
There exists an entire function $U\not\equiv 0$ with $\rho(U)<\alpha$ such that $J=GU$ is entire, and
    \begin{equation}\label{e5}
    FJ=FGU=-U(F''-t(F'+B)+B').
    \end{equation}
It is clear that $\rho(G)<\alpha$, so that $\rho(J)<\alpha$.

We easily obtain
	\begin{equation*}
	F^2+tF/2 = -T-B-K,
    \end{equation*}
where $K=F'-tF/2=G/2$, and so
    \begin{equation}\label{L}
	L^2=-T-B+M,
	\end{equation}
where $L=F+t/4$ and $M=t^2/16-K$. The functions $K,L,M$ are meromorphic and of order $\leq\alpha$.
It is clear that $L\not\equiv 0$ because $f$ is of infinite order, while $T$ is of finite order. A differentiation of \eqref{L} gives
	$$
	2LL' = -B'+M'-T'=-B'+M'+t(F'+F^2+B),
    $$
followed by a simple manipulation
    \begin{equation}\label{simple}
	L(2L'-tL)=-B'+M'+t(B-M).
    \end{equation}

Set $N=2L'-tL$. The crux of the proof is to show that
    \begin{equation}\label{crux}
    N\equiv 0.
    \end{equation}
Suppose on the contrary to this that $N\not\equiv 0$.
The poles of $N$ are among the zeros of $f$ and $T$. Hence $\lambda(1/N)<\alpha$, and so there exists
an entire function $Q\not\equiv 0$ of order $\leq \alpha$ such that $S=NQ$ is entire. Moreover, \eqref{simple} yields
    \begin{equation}\label{LS}
    \begin{split}
    LS &= LNQ=-Q(B'-M'+t(M-B))\\
    &=-Q\left(\left(\frac{B'}{B}-t\right)B-M'+tM\right).
    \end{split}
    \end{equation}
We proceed similarly as above. First, we find that there exists a constant $N_1>0$ and an $R$-set $E_3\supset E_1\cup E_2$
surrounding the zeros of $f$ and $T$ such that
 	$$
    \left|\frac{F''(z)}{F(z)}\right|+\left|\frac{F'(z)}{F(z)}\right|+\left|\frac{L'(z)}{L(z)}\right|
    =O\left(|z|^{N_1}\right),\quad z\not\in E_3,
	$$
and
	$$
    |\log |H(z)||+\log^+|M(z)|+\log^+|M'(z)|\leq |z|^\beta,\quad z\not\in E_3,
    $$
where $\alpha<\beta<\gamma<n$. Second, we find sectorial estimates for $L$ and $S$.
For almost all $\theta$ for which $h_T(\theta)>0$, we have
	\begin{equation}\label{T}
	\Re(\zeta e^{in\theta})r^n-O\left(r^\alpha\right)\leq \log |T(re^{i\theta})|
	\end{equation}
by \eqref{e2} and then $2h_L(\theta)=h_T(\theta)$ by \eqref{L}.  Dividing \eqref{LS} by $L$,
we have by \eqref{B} and \eqref{T} that
	\begin{equation}\label{SS}
    \begin{split}
	\log |S(re^{i\theta})| &\leq  (s(\theta)-1/2)\Re(\zeta e^{in\theta})r^n+O\left(r^\beta\right)\\
	&\leq  (s(\theta)-1/2)\Re(\zeta e^{in\theta})r^n/2<0
    \end{split}	
    \end{equation}
for all $r$ sufficiently large.
In other words, $h_S(\theta)<0$ for almost all $\theta$ for which $h_T(\theta)>0$. Similarly,
for almost all $\theta$ for which $h_T(\theta)<0$, we have
    	$$
	\log^+|L(re^{i\theta})|=O\left(r^\beta\right)
	$$
by \eqref{B} and \eqref{L}, and
	\begin{eqnarray*}
	\log^+|S(re^{i\theta})|&\leq& \log^+|N(re^{i\theta})|+O\left(r^\beta\right)\\
	&\leq& \log^+\left|2\frac{L'}{L}\cdot L\right|+\log^+|tL|+O\left(r^\beta\right)=O\left(r^\beta\right).
	\end{eqnarray*}
Thus the Phragm\'en-Lindel\"of principle again implies that $\rho(S)\leq \beta$, and hence $\rho(N)\leq \beta$.
Therefore we may find an $R$-set $E_4\supset E_3$ such that
	$$
	|\log |S(z)||\leq |z|^\gamma,\quad z\not\in E_4.
	$$
But this contradicts the fact that \eqref{SS} holds for $\theta$ in a set of positive measure.

We deduce by \eqref{crux} that $L=cT^{1/2}$ for some $c\in\C\setminus\{0\}$. Hence
	$$
	F=cT^{1/2}-t/4\quad\textrm{and}\quad F'=ctT^{1/2}/2-t'/4,
	$$
which together with \eqref{e3} imply
	$$
	-T-B=F'+F^2=c^2T+t^2/16-t'/4.
	$$
This yields $c^2=-1$ and $B=-t^2/16+t'/4$, which proves \eqref{B2}.

If $H$ has a zero of multiplicity $m\geq 1$ at some point $\zeta_0$, then $T(z)=a_m(z-\zeta_0)^m(1+o(1))$  and
	$$
	\frac{T'(z)}{T(z)}=\frac{m}{z-\zeta_0}(1+o(1))
	$$
near $\zeta_0$. But then $B$ has a double pole at $\zeta_0$ by \eqref{B2}, which is a contradiction. Hence $H$
has no zeros. Thus $P(z)=\zeta z^n+\log H(z)$ is a polynomial of degree $n$. The proof of the fact that
\eqref{lde2} has a zero-free solution base $\{f_1,f_2\}$ is similar to the reasoning in \cite[p.~8]{BLL} or \cite[p.~356]{Bank-Laine}.
If $g$ is any linear combination of $f_1,f_2$, then $\lambda(g)=\infty$. Thus, if a nontrivial solution $f$
of \eqref{lde2} satisfies $\lambda(f)<n$, it must be a constant multiple of one of $f_1,f_2$, and hence
has no zeros. \hfill$\Box$

\section{Oscillation in sectors}\label{angularradial-sec}

Instead of oscillation, Steinbart takes a step in the opposite direction by proving three results
on non-oscillation in the case when the coefficient $A(z)$ is an exponential polynomial \cite{Sbart}.
Under certain assumptions, the equation \eqref{lde2} is proven to be non-oscillatory in a sector,
that is, every non-trivial solution of \eqref{lde2} has at most finitely many zeros in the sector in
question. The main results in \cite{Sbart} are quite
technical, and the proofs are very involved relying on Strodt's theory.

Assuming that the coefficient $A(z)$ is entire and sufficiently small in an unbounded quasidisc $D$,
Hinkkanen and Rossi proved \cite{HR} that any nontrivial solution of \eqref{lde2} has at most one
zero in $D$. As a consequence of their main result, if $A(z)=P(z)e^{Q(z)}$, where $P(z),Q(z)$ are polynomials,
one can usually take $D$ to be a sector of opening $\pi/\deg(Q)$.

The next result is stated for an entire coefficient $A(z)$ of finite order, but the assumptions are
quite typical properties of exponential polynomials. The proof is rather elementary and relies on a
well known univalence criterion due to Nehari.

\begin{theorem}\label{fewzeros-thm}
Let $A(z)$ be an entire function of order $\rho(A)=\rho\in [1,\infty)$. Suppose that there exists two zeros
$\alpha,\beta\in\R$ of the indicator function $h_A$ such that $0<\beta-\alpha\leq\pi/\rho$ and
$h_A(\theta)<0$ for $\theta\in (\alpha,\beta)$. Then, for
any $\veps>0$, every nontrivial solution $f$ of \eqref{lde2} has at most finitely many zeros in the
sector $\alpha+\veps< \arg(z)< \beta-\veps$. The number of zeros may depend on $\veps$.
\end{theorem}

\begin{remark}
In fact, the result holds if $A(z)$
only tends to zero exponentially in some sector $S$ as in \eqref{A-estimate}, but independently on the order
of $A(z)$ or the indicator of $A(z)$. However, the authors feel that the current statement is more natural
due to two reasons: (a) Functions of order $<1/2$ do not have this property by the $\cos \pi\rho$-theorem.
(b) If an entire function $A(z)$ is of order $\rho<1$ and tends to zero in an angle of opening $\pi/\rho$, then it must vanish
identically by theorems due to Phragm\'en-Lindel\"of and Liouville.
\end{remark}

\bigskip
\noindent
\emph{Proof of Theorem~\ref{fewzeros-thm}.}
Let $\veps\in (0,(\beta-\alpha)/2)$, and recall the fundamental relation of the indicator from \cite[p.~53]{Levin1}:
    $$
    h_A(\theta)\leq
    \frac{h_A(\theta_1)\sin \rho (\theta_2-\theta)+h_A(\theta_2)\sin\rho (\theta-\theta_1)}{\sin \rho(\theta_2-\theta_1)},
    $$
where $\theta_1<\theta<\theta_2$ and $0<\theta_2-\theta_1<\pi/\rho$. Choose $\theta_1=\alpha+\veps/2$ and $\theta_2=\beta-\veps/2$,
and note that $h_A(\theta_1)<0$ and $h_A(\theta_2)<0$. We conclude that there exists a constant $\delta_0>0$ such that
$h_A(\theta)<-\delta_0$ for every $\theta\in (\alpha+\veps,\beta-\veps)$. By the Phragm\'en-Lindel\"of
theorem for an angle, there exist constants $\delta\in (0,\delta_0)$ and $R>0$ such that
    \begin{equation}\label{A-estimate}
    |A(z)|\leq \exp(-\delta |z|^\rho),\quad z\in S(\alpha+\veps,\beta-\veps,R),
    \end{equation}
where $S(a,b,R)=\{z\in\C : a<\arg(z)<b,\, |z|>R\}$.

Next we construct a conformal mapping $\Phi$ from $\D$ onto $S(\alpha+\veps,\beta-\veps,R)$
in two parts. First we construct a conformal mapping $T$ from $\D$ onto a lens-shaped domain $D$ symmetric
with respect to the $x$-axis and bounded by an arc of $\partial\D$ and by a Poincar\'e line traveling through a point $s\in (0,1)$.
An exact formulation of the mapping $T$ follows by the discussions in \cite[pp.~208-209]{Nehis}.
Alternatively, we may construct $T$ as a composition of elementary transformations as follows:
First we map $\D$ onto the right half-plane $H$ by means of $z\mapsto (1+z)/(1-z)$. Then the principal branch
of the square root maps $H$ onto the sector $|\arg(z)|<\pi/4$. We map this back into the unit disc and
onto a lens shaped domain with straight angles at the points $z=\pm 1$ by means of $z\mapsto (z-1)/(z+1)$.
Finally, after a suitable dilatation $z\mapsto tz$ for some $t\in (0,1)$, rotation $z\mapsto iz$, and a translation
$z\mapsto z+u$ for some $u\in (0,1)$, we reach our target domain $D$. This gives us
\begin{equation}\label{conformal-map}
    T(z)=\frac{(u+it)\sqrt{1+z}+(u-it)\sqrt{1-z}}{\sqrt{1+z}+\sqrt{1-z}},
    \end{equation}
where $u=(1+s)/2$ and $t=(1+\sqrt{2})(1-s)/2$. Second, a conformal mapping $L$
from $D$ onto $S(\alpha+\veps,\beta-\veps,R)$ is of the form
    \begin{equation}\label{confmap}
    L(z)=e^{i\varphi}\left(\frac{1+z}{1-z}\right)^\gamma,
    \end{equation}
where $\gamma=(\beta-\alpha-2\veps)/\pi<(\beta-\alpha)/\pi\leq 1/\rho\leq 1$ and $\varphi=(\beta-\alpha)/2$.
Now $\Phi=L\circ T$ is the mapping we are looking for.

The Schwarzians of $L$ and $T$ can be computed directly (by computer).
Alternatively, we may rely on the discussions in \cite[pp.~208-209]{Nehis}. We have
    $$
    S_T(z) = \frac{3}{2(1-z^2)^2}\quad\textrm{and}\quad
    S_L(z) = \frac{2(1-\gamma^2)}{(1-z^2)^2}.
    $$
Note in particular that the Schwarzians are independent on the constants $u,t,\varphi$.
It follows that
    \begin{eqnarray*}
    S_\Phi(z) &=& S_L(T(z))T'(z)^2+S_T(z)\\
    &=&\frac{2(1-\gamma^2)}{(1-T(z)^2)^2}T'(z)^2+\frac{3}{2(1-z^2)^2},
    \end{eqnarray*}
where
    $$
    |T'(z)|^2=\frac{4t^2}{|2+2\sqrt{1-z^2}|^2|1-z^2|}\leq \frac{1}{1-|z|^2},\quad z\in\D.
    $$
If $|1-z^2|< 1/2$, we use the continuity and the construction of $T(z)$ together with $T(-i)=1$
to conclude that $|1-T(z)^2|>\delta>0$. Thus
    $$
    (1-|z|^2)^2|S_\Phi(z)|\leq 2(1-\gamma^2)(1-|z|^2)/\delta^2+3/2.
    $$
If $|1-z^2|\geq 1/2$, we use the Schwarz-Pick lemma $(1-|z|^2)|T'(z)|\leq 1-|T(z)|^2$, and obtain
    $$
    (1-|z|^2)^2|S_\Phi(z)|\leq 2(1-\gamma^2)+6(1-|z|^2)^2.
    $$
Hence, in all cases there exists a constant $x_0\in (0,1)$ such that
    $$
    (1-|z|^2)^2|S_\Phi(z)|<2,\quad x_0<|z|<1.
    $$

Let $\{f_1,f_2\}$ be a fundamental solution base of \eqref{lde2}. Define $h=f_1/f_2$, so that
$S_h(z)=2A(z)$. The Schwarzian of $g=h\circ \Phi$ is given by
    $$
    S_g(z)=2A(\Phi(z))\Phi'(z)^2+S_\Phi(z).
    $$
Since $\Phi$ is univalent, we have $|\Phi'(z)|=O\left(1/(1-|z|)^3\right)$.
Hence, if $R>0$ is large enough, or, in other words, if $s\in (0,1)$ is close enough to one, then
\eqref{A-estimate} yields
    $$
    (1-|z|^2)^2|S_g(z)|\leq 2,\quad x_0<|z|<1.
    $$
By an extension of classical Nehari's theorem \cite[Corollary~1]{Schwarz}, we find that $g$ is finitely valent in $\D$.
In other words, $h$ is finitely valent in $S(\alpha+\veps,\beta-\veps,R)$, that is, any solution of \eqref{lde2} has at most
finitely many zeros in $S(\alpha+\veps,\beta-\veps,R)$. Hence the assertion follows. \hfill$\Box$

\begin{example}
If $A(z)=e^z+e^{2z}+e^{3z}$, then $h_A(\theta)=\cos\theta <0$ for all $\theta\in(\pi/2,3\pi/2)$,
and hence the solutions of \eqref{lde2} have at most finitely many zeros in the (slightly squeezed)
left half-plane.
\end{example}

\begin{example}
Let $A(z)$ be a solution of a differential equation $g''+P(z)g=0$, where $P(z)=a_nz^n+\cdots +a_0$ is a
polynomial of degree $n$. Then $\rho(A)=(n+2)/2$ by \cite[Proposition~5.1]{Laine}. Moreover, by Hille's method of
asymptotic integration \cite[Section 7.4]{Hil}, it is known that $A(z)$ either blows up or converges
to zero exponentially in sectors determined by the critical rays $\arg (z)=\theta_j$, where
    $$
    \theta_j=\frac{1}{n+2}(2\pi j-\arg (a_n)),\quad j=0,\ldots ,n+1.
    $$
Since $\theta_{j+1}-\theta_j=2\pi/(n+2)=\pi/\rho(A)$, we find that the assumptions in Theorem~\ref{fewzeros-thm}
are quite typical for solutions $A(z)$ of $g''+P(z)g=0$. If $n$ is an odd integer, then $A(z)$ is certainly
not an exponential polynomial.
\end{example}

\begin{theorem}\label{real-solutions}
Let $P(z)=a_nz^n+\cdots +a_0$ be a polynomial with real coefficients and $a_n>0$. Define $A(z)=P(e^z)$.
Then the number of zeros $n([0,r),f)$ of a nontrivial real solution $f$ of \eqref{lde2} on the interval
$[0,r)$ satisfies $n([0,r),f)\geq Ce^{nr/2}$ for some constant $C>0$.
\end{theorem}

\begin{proof}
Define $q(x)=\frac{a_n}{2}e^{nx}$ for $x\in [0,\infty)$. Then $q(x)$ and $q'(x)$ are continuous with
$q'(x)q(x)^{-3/2}\to 0$ as $x\to\infty$. Using either \cite[Corollary~5.3]{Hart} or \cite[Theorem 9.5.1]{Hil},
we conclude the following: The number of zeros $n_0([0,x),y)$ of a non-trivial solution $y$ of $y''+q(x)y=0$
on the interval $[0,x)$ for large $x$ satisfies
    \begin{eqnarray*}
    n_0([0,x),y) &\sim& \frac{1}{\pi}\int_0^x q(t)^{1/2}\, dt\geq \frac{1}{\pi}\int_{x-1}^xq(t)^{1/2}\, dt\\
    &=& \frac{\sqrt{2a_n}}{n\pi}\left(1-e^{-n/2}\right)e^{nx/2}\geq \frac{\sqrt{2a_n}\cdot3}{10n\pi}e^{nx/2}.
    \end{eqnarray*}
On the other hand, $A(z)$ is real-valued on the real axis, and $A(x)\geq q(x)$ for $x\geq r_0$. Thus
all solutions of \eqref{lde2} are constant multiples of real entire functions. By the
standard Sturm comparison theorem, any real solution $f$ vanishes at least once between any two zeros of $y$.
Hence the magnitude of $n([0,x),f)$ is at least that of $n_0([0,x),y)$. This yields the assertion.
\end{proof}

\begin{example}
If $P(z)=-z^2-z$, then
\eqref{lde2} with $A(z)=P(e^z)$ has a zero-free solution $f(z)=\exp\left(e^z\right)$.
Hence the assumption $a_n>0$ in Theorem~\ref{real-solutions} is essential.
\end{example}

Suppose that $A(z)$ is an exponential polynomial, and $\theta^*$
is such that $h_A(\theta^*)=\max_\theta h_A(\theta)$ $(>0)$. We prove that most
solutions of \eqref{lde2} have an infinite exponent of convergence in an arbitrarily small sector
    $$
    S(\theta^*,\veps)=\{z\in\C : |\arg(z)-\theta^*|<\veps\}.
    $$
This is in contrast to the situation in Theorem~\ref{fewzeros-thm}.

\begin{theorem}\label{sector-thm}
Let $A(z)$ be an exponential polynomial of order $n$, let $\theta^*\in\R$ be such that $h_A(\theta^*)=\max_\theta h_A(\theta)$,
and let $\veps \in (0,\pi/2n)$. Moreover, suppose that \eqref{lde2} has a solution base $\{f_1,f_2\}$, and let $f_3$ be any linear
combination of $f_1,f_2$ but not of the form $cf_j$. Then
    $$
    \max_{j=1,2,3}\{\lambda_2(S(\theta^*,\veps),f_j)\}=n,
    $$
where $\lambda_2(S(\theta^*,\veps),f_j)$ is the hyper-exponent of convergence of the zeros of $f_j$ in
the sector $S(\theta^*,\veps)$.
\end{theorem}

\begin{proof}
For simplicity, we may suppose that $f_3=f_1-f_2$. Let $F=f_1/f_2$. Then the zeros, poles and $1$-points of $F$ correspond to the zeros of $f_1,f_2,f_3$,
respectively. Moreover, $2A(z)=S_F(z)$, where
    $$
    S_F(z)=\frac{F'''(z)}{F'(z)}-\frac32\left(\frac{F''(z)}{F'(z)}\right)^2
    $$
is the Schwarzian derivative of $F$. For brevity, we indicate the angular domain $S(\theta^*,\veps)$ by
subindex $S$, and denote the angular Tsuji characteristic by $T_S(r,g)=m_S(r,g)+N_S(r,g)$, where
$g$ is any meromorphic function. See \cite[p.~58]{Zheng} for the definitions of these functions,
but with different notation.

Recall from \cite[p.~56]{Levin1} that
    $$
    h_A(\theta)\geq h_A(\theta^*)\cos \left(n(\theta-\theta^*)\right)
    $$
for $|\theta-\theta^*|\leq \pi/n$. Hence, for our choice of $\veps\in (0,\pi/2n)$, we have $h_A(\theta)>0$ for all $\theta\in [\theta^*-\veps,\theta^*+\veps]$.
Since $A(z)$ grows fast in the sector $S(\theta^*,\veps)$, we conclude by the angular second main theorem \cite[p.~59]{Zheng}
and by the angular lemma on the logarithmic derivative \cite[pp.~85-86]{Zheng}, that
    \begin{eqnarray*}
    Cr^n &\leq& m_S(r,A)\leq m_S\left(r,\frac{F'''}{F'}\right)
    +2m_S\left(r,\frac{F''}{F'}\right)+O(1)\\
    &=& O\left(\log^+ T_S(2r,F)\right)+O(\log r)\\
    &\leq& \log^+N_S(2r,\infty,F)+\log^+N_S(2r,0,F)+\log^+N_S(2r,1,F)\\
    &&+O(\log^+\log^+T_S(2r,F))+O(\log r)\\
    &\leq & 3\max_{j=1,2,3}\{\log^+N_S(2r,0,f_j)\}+O(\log^+\log^+T_S(2r,F))+O(\log r)
    \end{eqnarray*}
outside of a possible exceptional set of $r$-values of finite linear measure. Since both sides of
this inequality are increasing functions, we may avoid this exceptional set via \cite[Lemma~1.1.1]{Laine}.
This proves $\max_{j=1,2,3}\{\lambda_2(S(\theta^*,\veps),f_j)\}\geq n$. Since the inequalities
    $$
    \max_{j=1,2,3}\{\lambda_2(S(\theta^*,\veps),f_j)\}
    \leq \max_{j=1,2,3}\{\lambda_2(f_j)\}\leq\max_{j=1,2,3}\{\rho_2(f_j)\}\leq n
    $$
are clear, the assertion follows.
\end{proof}

\noindent
\textbf{Acknowledgement.}
The authors want to thank James Langley for helpful discussions. In particular,
Theorem~\ref{referee-thm} is essentially due to Langley.

\medskip
\noindent
\emph{Janne Heittokangas}\\
\textsc{Taiyuan University of Technology}\\
\textsc{Faculty of Mathematics}\\
\textsc{Yingze West Street, No.~79, Taiyuan 030024, China}\\
\texttt{email:janne.heittokangas@uef.fi}

\medskip
\noindent
\emph{Ilpo Laine \& Janne Heittokangas}\\
\textsc{University of Eastern Finland}\\
\textsc{Department of Physics and Mathematics}\\
\textsc{P.O.~Box 111, 80101 Joensuu, Finland}\\
\texttt{email:ilpo.laine@uef.fi}

\medskip
\noindent
\emph{Katsuya Ishizaki}\\
\textsc{The Open University of Japan}\\
\textsc{2-11 Wakaba Mihama-ku Chiba}\\
\textsc{261-8586, Japan}\\
\texttt{email:ishizaki@ouj.ac.jp}

\medskip
\noindent
\emph{Kazuya Tohge}\\
\textsc{Kanazawa University}\\
\textsc{College of Science and Engineering}\\
\textsc{Kakuma-machi, Kanazawa 920-1192, Japan}\\
\texttt{email:tohge@se.kanazawa-u.ac.jp}

\end{document}